\documentclass[12pt]{article}
\usepackage{mathrsfs}
\usepackage{tipa}
\usepackage{amssymb}
\usepackage{amsfonts}
\usepackage{cite}
\usepackage{graphicx}
\usepackage{latexsym,amsmath,amssymb,amsfonts,amsthm}

\newtheorem{theorem}{Theorem}[section]
\newtheorem{lemma}[theorem]{Lemma}

\newtheorem*{theorem a}{Theorem A}

\newtheorem*{conjecture1}{Rigidity Conjecture}

\begin{document}
\setcounter{page}{1}
\title{Horizontal diameter of unit spheres with polar foliations and infinitesimally polar actions}
\author{Yi Shi}
\date{}
\protect
\maketitle ~~~\\[-5mm]

\protect\footnotetext{\!\!\!\!\!\!\!\!\!\!\!\!\! {\bf MSC 2010:}
53C12, 53C20.
}
\maketitle ~~~\\[-5mm]
{\footnotesize School of Mathematics and Statistics, Guizhou University of Finance and Economics, Guiyang 550025, China, e-mail: shiyi\underline{\hbox to 0.2cm{}}math@163.com}\\[1mm]

{\bf Abstract:} \noindent For a singular Riemannian foliation $\mathcal{F}$ on a Riemannian manifold, a curve is called horizontal if it meets the leaves of $\mathcal{F}$ perpendicularly. For a singular Riemannian foliation $\mathcal{F}$ on a unit sphere $\mathbb{S}^{n}$, we show that if $\mathcal{F}$ is a polar foliation or if $\mathcal{F}$ is given by the orbits of an infinitesimally polar action, then the horizontal diameter of $\mathbb{S}^{n}$ is $\pi$, i.e., any two points in $\mathbb{S}^{n}$ can be connected by a horizontal curve of length $\leq\pi$.
\markright{\sl\hfill \hfill}\\

\section{Introduction}
%\label{Introduction}

\renewcommand{\thesection}{\arabic{section}}
\renewcommand{\theequation}{\thesection.\arabic{equation}}
\setcounter{equation}{0}

A \emph{singular Riemannian foliation} $\mathcal{F}$ on a Riemannian manifold $M$ is a decomposition of $M$ into smooth injectively immersed submanfolds $L(x)$, called leaves, such that it is a singular foliation and any geodesic starting orthogonally to a leaf remains orthogonal to all leaves it intersects. A leaf $L$ of $\mathcal{F}$ (and each point in $L$) is called \emph{regular} if the dimension of $L$ is maximal, otherwise $L$ is called \emph{singular}; see \cite{wa,mol,ra3}.

A singular Riemannian foliation is called a \emph{polar foliation} if, for each regular point $p$, there is a totally geodesic complete immersed submanifold $\Sigma_{p}$, called \emph{section}, that passes through $p$ and that meets each leaf orthogonally. A typical example of a polar foliation is the partition of a Riemannian manifold into parallel submanifolds to an isoparametric submanifold $L$ in a Euclidean space. Recall that a submanifold $L$ of a Euclidean space is called \emph{isoparametric} if its normal bundle is flat and the principal curvatures along any parallel normal vector field are constant; see \cite{al1,al2,at,hpt,te,th1,th2}.

For a singular Riemannian foliation $\mathcal{F}$ on a Riemannian manifold $M$, a curve is called \emph{horizontal} if it meets the leaves of $\mathcal {F}$ perpendicularly. When $M$ has positive curvature, Wilking \cite{wi} proved that any two points in $M$ can be connected by a piecewise smooth horizontal curve. Thus Wilking introduced the \emph{horizontal metric} $g_{\mathcal{H}}$ on $M$ by defining the \emph{horizontal distance} of two points as the infimum over the length of all horizontal curves connecting these two points. Besides the intrinsic interest in such object, one reason to study $g_{\mathcal{H}}$ is its connection with sub-Riemannian geometry; see \cite{mon,sx}. It is natural to define the \emph{horizontal diameter} $diam_{\mathcal{H}}M$ of $M$ by $$diam_{\mathcal{H}}M=\sup\{d_{\mathcal{H}}(p, q)\mid\ p,q\in M\},$$ where $d_{\mathcal{H}}(p, q)$ is the horizontal distance of $p$ and $q$. Notice that $diam_{\mathcal{H}}M \geq diam(M)$, where $diam(M)$ is the diameter of $M$ defined by its Riemannian metric.

Recently a lot of progress has been made in the singular Riemannian foliations of round spheres(\cite{go,gl1,gl2,gl3,gr,lr,lw,mr,ra1,ra2}). In \cite{sx} we studied the horizontal diameter rigidity of unit sphere $\mathbb{S}^{n}$. We proved that for many classes of singular Riemannian foliations on $\mathbb{S}^{n}$, the horizontal diameter of $\mathbb{S}^{n}$ is $\pi$, i.e., any two points in $\mathbb{S}^{n}$ can be connected by a horizontal curve of length $\leq\pi$. Based on our results, we also proposed the following rigidity conjecture:
\begin{conjecture1}
For any singular Riemannian foliation on a unit sphere $\mathbb{S}^{n}$, we have $diam_{\mathcal{H}}\mathbb{S}^{n}=\pi$.
\end{conjecture1}

We call this conjecture ``rigidity" since it asserts that the inequality $diam_{\mathcal{H}}\mathbb{S}^{n} \geq diam(\mathbb{S}^{n})=\pi$ should be an equality. In this paper we prove this conjecture for the case of polar foliation:

\begin{theorem}\label{1.1}
For any polar foliation on a unit sphere $\mathbb{S}^{n}$, we have $diam_{\mathcal{H}}\mathbb{S}^{n}=\pi$.
\end{theorem}

Notice that Theorem 1.1 is not new. In fact, by \cite{th1} polar foliations of unit sphere $\mathbb{S}^{n}$ with codimension $\geq 2$ give rise to spherical buildings, so the horizontal diameter of $\mathbb{S}^{n}$ must be $\pi$. On the other hand, the case of codimension 1 had been proved by \cite{es}. However, in this paper we will give a short and uniform proof.

Recall that a \emph{Riemannian orbifold} is a metric space locally isometric to quotients of Riemannian manifolds by finite groups of isometries.
It has been shown by Lytchak and Thorbergsson \cite{lt} that for an isometric action of a compact Lie group $G$ on a Riemannian manifold $M$, the quotient $X=M/G$ is a Riemannian orbifold if and only if all slice representations of the action are polar. In \cite{lt} an isometric action of $G$ on $M$ with this property has been called \emph{infinitesimally polar}.

In \cite{gl2} Gorodski and Lytchak classified all representations of compact connected Lie groups $G$ on Euclidean spaces whose induced action on the unit sphere $\mathbb{S}^{n}$ has the orbit space $X=\mathbb{S}^{n}/G$ isometric to a Riemannian orbifold. Here we cite a slightly simplified version of their classification as follow:

\begin{theorem}\label{GL} (Gorodski and Lytchak \cite{gl2}) Let a compact connected Lie group $G$ act effectively and isometrically on the unit sphere $\mathbb{S}^{n}$. The quotient $X=\mathbb{S}^{n}/G$ is a Riemannian orbifold if and only if one of the following cases occurs:

\noindent(1) The action of $G$ is polar.

\noindent(2) $G$ has rank $1$ and acts almost freely on $\mathbb{S}^{n}$.

\noindent(3) $G$ has cohomogeneity $2$.

\noindent(4) $G$ is one of the actions listed in Table 1 of \cite{gl2}. In particularly, $diam X=\frac{\pi}{2}$, $X$ has constant curvature $4$ and the dimension $\dim X$ of $X$ satisfies $3\leq\dim X\leq5$.
\end{theorem}

Thus for $G$ in Theorem \ref{GL}, if $G$ has rank at least $2$ and the action of $G$ is not polar, then $2\leq\dim X\leq5$. As been pointed out in \cite{gl2}, this result is rather surprising if one compares it with the examples of Riemannian orbifolds arising as quotients of non-homogeneous singular Riemannian foliations on unit spheres. By using representations of Clifford algebras Radeschi \cite{ra2} showed that there are singular Riemannian foliations (the so-called \emph{Clifford foliations}) on unit spheres $\mathbb{S}^{n}$, for $n$ large enough, such that the quotient space is a Riemannian orbifold and isometric to a hemisphere of curvature $4$, which can be of arbitrary large dimension!

In \cite{sx} we proved the above rigidity conjecture of horizontal diameter for the class of Clifford foliations constructed by Radeschi \cite{ra2}. As an application of Theorem \ref{1.1}, the following result of this paper continues this investigation:
\begin{theorem}\label{infinitesimally polar} Let a compact connected Lie group $G$ act effectively and isometrically on the unit sphere $\mathbb{S}^{n}$ and $\mathcal{F}$ be the singular Riemannian foliation given by the orbits of the $G$ action. If the quotient $X=\mathbb{S}^{n}/G$ is a Riemannian orbifold (or equivalently the action of $G$ is infinitesimally polar), then $diam_{\mathcal{H}}\mathbb{S}^{n}=\pi$.
\end{theorem}

To prove Theorem \ref{infinitesimally polar}, we will use Straume's results \cite{st} on representations of compact Lie groups of cohomogeneity $\leq3$. We show that these results can be stated in the following form, which is applicable in the proof of Theorem \ref{infinitesimally polar}:

\begin{theorem}\label{cohom 2}
Let a compact connected Lie group $G$ act effectively and isometrically on the unit sphere $\mathbb{S}^{n}$. If $G$ has cohomogeneity $\leq 2$, then the action of $G$ is either polar or reducible.
\end{theorem}

In Section 2 we recall basics about polar foliations on $\mathbb{R}^{n+1}$ and $\mathbb{S}^{n}$, then we prove Theorem 1.1. In Section 3 we prove the Lemma \ref{reducible} concerned with reducible actions on unit sphere, then we use it to prove Theorem \ref{cohom 2} and \ref{infinitesimally polar}.\\

\noindent$\mathbf{Acknowledgement.}$ I owe especial thanks to Marco Radeschi for many discussions that have been very helpful in clarifying my ideas. I am very grateful to Gudlaugur Thorbergsson for many discussions on Bott-Samelson cycle. I would like to thank my advisor Fuquan Fang for his supports and valuable discussions. Part of this paper was completed during a visit at the Capital Normal University, I am very grateful to Zhenlei Zhang for the invitation to Capital Normal University. For helpful discussions and comments I would like to thank Chao Qian and Yueshan Xiong. The author is supported by National Natural Science Foundation of China No.11801337 and the Introducing Talents Project in Guizhou University of Finance and Economics (2019YJ042).

I would like to thank the referee for the careful review and critical comments that helped to clarify some issues and simplify the proof of Theorem 1.1 considerably.

\section{Polar foliations on $\mathbb{R}^{n+1}$ and $\mathbb{S}^{n}$}
\renewcommand{\thesection}{\arabic{section}}
\renewcommand{\theequation}{\thesection.\arabic{equation}}
\setcounter{equation}{0}
\setcounter{theorem}{0}

Let $\mathcal{F}$ be a polar foliation of a Euclidean space $\mathbb{R}^{n+1}$ or of a unit sphere $\mathbb{S}^{n}$. It was observed by Alexandrino in \cite{al1} that $\mathcal{F}$ is an isoparametric foliation in these spaces. Conversely, it is clear from Terng's paper \cite{te} that isoparametric foliations are polar (see \cite{te} for the definition of isoparametric foliation). Below we state certain properties of isoparametric foliations which are proved in \cite{te}, see also \cite{hpt,th1,th2}.

Suppose $\mathcal{F}$ is an isoparametric foliation of $\mathbb{R}^{n+1}$ or of $\mathbb{S}^{n}$. Then $\mathcal{F}$ have closed leaves and there are no exceptional leaves, i.e., that all regular leaves have trivial holonomy. A regular leaf $M$ of $\mathcal{F}$ is an isoparametric submanifold in $\mathbb{R}^{n+1}$ or in $\mathbb{S}^{n}$. This means $M$ has flat normal bundle $VM$ and the principal curvatures in the direction of any parallel normal vector field are constant. The leaves are in one-to-one correspondence with the images of the endpoint maps of the parallel normal vector fields over $M$, so all the leaves are parallel manifolds of $M$.

Let $\mathcal{F}$ be an isoparametric foliation of $\mathbb{S}^{n}$. Then the foliation of $\mathbb{R}^{n+1}$ whose leaves are homothetic to leaves of $\mathcal{F}$ is also isoparametric. Conversely, assume that $\mathcal{F}$ is an isoparametric foliation of $\mathbb{R}^{n+1}$ with compact leaves. Then there is a focal submanifold that coincides with a point. We can assume that this point is the origin $o$. Then the leaves meeting $\mathbb{S}^{n}$ lie in $\mathbb{S}^{n}$ and induce an isoparametric foliation there. In this paper, we only consider isoparametric foliations of $\mathbb{R}^{n+1}$ with compact leaves.

Let $\mathcal{\overline{F}}$ be a polar or isoparametric foliation of $\mathbb{R}^{n+1}$ whose leaves are compact with codimension $\geq2$. A regular leaf $M$ of $\mathcal{\overline{F}}$ is a compact isoparametric submanifold in $\mathbb{R}^{n+1}$. Since a compact isoparametric submanifold is included in a round sphere, without loss of generality we can assume that it lies in the unit sphere $\mathbb{S}^{n}$ of $\mathbb{R}^{n+1}$ with center at the origin $o$. We will consider the normal spaces of $M$ to be affine subspaces of $\mathbb{R}^{n+1}$ that we denote by $x + V_{x}M$ for $x \in M$. Every $x + V_{x}M$ is a linear subspace of $\mathbb{R}^{n+1}$ since $M\subset\mathbb{S}^{n}$. Notice that $x + V_{x}M$ is the section passing $x$ when we view $\mathcal{\overline{F}}$ as a polar foliation of $\mathbb{R}^{n+1}$.

The tangent bundle $TM$ has a canonical splitting as the orthogonal direct sum of $g$ subbundles $E_{i}$, each of which is integrable, say of fiber dimension $m_{i}$, and has associated to it a canonical parallel normal field $v_{i}$ (the ith \emph{curvature normal vector}). These are related and characterized by the following fact: the set $\{A_{v}\}$ of shape operators, $v\in V_{x}M$, is a commuting family of self-adjoint operators on $T_{x}M$ and the $E_{i}(x)$ are the common eigenspaces with corresponding eigenvalues $\langle v, v_{i}(x)\rangle$. These $E_{i}$'s are called the \emph{curvature distributions} of $M$ and the $m_{i}$'s are called the \emph{multiplicities} of M. The leaf $S_{i}(x)$ of the integrable subbundle $E_{i}$ through $x\in M$ is a standard $m_{i}$ dimensional sphere in $\mathbb{R}^{n+1}$ and the center of $S_{i}(x)$ is $x + (v_{i}(x)/\langle v_{i}, v_{i}\rangle)$.

Now fix a point $x\in M$. Then the focal set of $M$, i.e., the set of critical values of the exponential map of $\mathbb{R}^{n+1}$ restricted to $VM$, intersected with the normal space $x + V_{x}M$, is the union over the $g$ linear hyperplanes $l_{i}(x):= v_{i}(x)^{\perp}$ in $x + V_{x}M $. In fact $(x + v_{0})\in l_{i}(x)$ if and only if $\langle v_{0}, v_{i}(x)\rangle = 1$. It turns out that the reflection $R_{i}^{x}$ of $x + V_{x}M$ in $l_{i}(x)$ permutes the hyperplanes $l_{1}(x),\ldots , l_{g}(x)$. Thus the group $W^{x}$ generated by the reflections $R_{1}^{x},\ldots, R^{x}_{g}$ is a finite Coxeter group. Since the holonomy of the normal bundle $VM$ is trivial, for each $x, y$ in $M$ we have a canonical parallel translation map $\pi_{x,y}:V_{x}M\rightarrow V_{y}M$. Because $\pi_{x,y}:V_{x}M\rightarrow V_{y}M$ conjugates $W^{x}$ to $W^{y}$, we have a well-defined Coxeter group $W$ associated to $M$. We call $W$ the \emph{Coxeter group} of the isoparametric submanifold $M$.

Let $U$ be a connected component of the complement of the union of the hyperplanes $l_{i}(x)$ in $x + V_{x}M$. Then its closure $\overline{U}$ is a simplicial cone, and a fundamental domain or \emph{chamber} for the Coxeter group $W$, i.e., each $W$-orbit meets $\overline{U}$ in exactly one point. $\overline{U}$ is also the leaf space of $\mathcal{\overline{F}}$, i.e., $\overline{U}=\mathbb{R}^{n+1}/\mathcal{\overline{F}}=(x + V_{x}M)/W$. Let $q\in x + V_{x}M$. Then $q$ is $W$-regular if and only if $q$ is nonfocal, i.e., $q$ is not in any $l_{i}(x)$.

Next we consider a polar or isoparametric foliation $\mathcal{F}$ of a unit sphere $\mathbb{S}^{n}$. Since $\mathcal{F}$ is an isoparametric foliation of $\mathbb{S}^{n}$, the foliation $\mathcal{\overline{F}}$ of $\mathbb{R}^{n+1}$ whose leaves are homothetic to leaves of $\mathcal{F}$ is also isoparametric. Thus $\mathcal{F}=\mathcal{\overline{F}}|_{\mathbb{S}^{n}}$ is the restriction of $\mathcal{\overline{F}}$ in $\mathbb{S}^{n}$. Let $x\in \mathbb{S}^{n}$ be a regular point of $\mathcal{F}$ and set $M:=L(x)$. The sections passing $x$ in $\mathcal{\overline{F}}$ and $\mathcal{F}$ are $\mathbb{R}^{k+1}_{x}:=(x + V_{x}M)$ and $\mathbb{S}^{k}_{x}:=\mathbb{R}^{k+1}_{x}\cap \mathbb{S}^{n}$ respectively, where $k\geq1$ is the codimension of $M$ in $\mathbb{S}^{n}$.  As the tangent bundle of an isoparametric submanifold $M$ of both $\mathbb{S}^{n}$ and $\mathbb{R}^{n+1}$, $TM$ has the same canonical splitting $TM=E_{1}\oplus \cdots \oplus E_{g}$ as the orthogonal direct sum of $g$ curvature distributions $E_{i}$ in both cases.

Set $\overline{l}_{i}(x)=l_{i}(x)\cap \mathbb{S}^{k}_{x}$ for $1\leq i\leq g$. As the focal set of $M$ in $\mathcal{\overline{F}}$ intersects the section $\mathbb{R}^{k+1}_{x}$ in the union of $l_{1}(x), \cdots, l_{g}(x)$, the focal set of $M$ in $\mathcal{F}$ intersects the section $\mathbb{S}^{k}_{x}$ in the union of $\overline{l}_{1}(x), \cdots, \overline{l}_{g}(x)$. If $y$ is not the origin $o$, and $y\in l_{i}(x)$ is a focal point of $M$ in $\mathcal{\overline{F}}$ corresponding to the curvature distributions $E_{i}$, then $\overline{y}\in \overline{l}_{i}(x)$ is a focal point of $M$ in $\mathcal{F}$ corresponding to the curvature distributions $E_{i}$, where $\overline{y}-o=(y-o)/\parallel y-o\parallel$.

As those reflections of $\mathbb{R}^{k+1}_{x}$ in $l_{1}(x),\cdots , l_{g}(x)$ generate a finite Coxeter group $W$ in $\mathcal{\overline{F}}$, these reflections of $\mathbb{S}^{k}_{x}$ in $\overline{l}_{1}(x),\cdots , \overline{l}_{g}(x)$ generate the same Coxeter group $W$ in $\mathcal{F}$. As $\overline{U}=\mathbb{R}^{n+1}/\mathcal{\overline{F}}=\mathbb{R}^{k+1}_{x}/W$ is both the leaf space of $\mathcal{\overline{F}}$ and a chamber on $\mathbb{R}^{k+1}_{x}$ for the Coxeter group $W$, $\overline{U}\cap \mathbb{S}^{n}=\mathbb{S}^{n}/\mathcal{F}=\mathbb{S}^{k}_{x}/W$ is both the leaf space of $\mathcal{F}$ and a chamber on $\mathbb{S}^{k}_{x}$ for the Coxeter group $W$.

By theories of $\mathcal{F}$ and $\mathcal{\overline{F}}$ above we are ready to prove Theorem \ref{1.1}.\\

\noindent {\it Proof of Theorem \ref{1.1}.} Let $\mathcal{F}$ be a polar foliation of $\mathbb{S}^{n}$. For any $p, q\in \mathbb{S}^{n}$, we will show that $d_{\mathcal{H}}(p, q)\leq \pi$. Since the union of all regular leaves are open and dense in $\mathbb{S}^{n}$, we can assume that both $L(p)$ and $L(q)$ are regular leaves. We set $M:=L(q)$. Let $\Sigma$ denote the section of $\mathcal F$ through $p$ and let $q_{0}$ denote the point in $M\cap\Sigma$ farthest away from $p$. Choose a (unit speed) minimal horizontal geodesic $\gamma(\theta)$ from $q_{0}:=\gamma(0)$ to $p:=\gamma(\theta_{0})$, and assume that $\gamma'(0)=\xi$. Extend $\xi$ to a parallel unit normal vector field $\xi$ of $M$.

Since $\mathcal{F}$ is an isoparametric foliation of $\mathbb{S}^{n}$, the foliation $\mathcal{\overline{F}}$ of $\mathbb{R}^{n+1}$ whose leaves are homothetic to leaves of $\mathcal{F}$ is also isoparametric. Thus $M$ is an isoparametric submanifold of both $\mathbb{S}^{n}$ and $\mathbb{R}^{n+1}$, and $TM$ has the same canonical splitting $TM=E_{1}\oplus \cdots \oplus E_{g}$ as the orthogonal direct sum of $g$ curvature distributions $E_{i}$ in both cases. For $x\in M$ and $1\leq i\leq g$, the leaf $S_{i}(x)$ of $E_{i}$ through $x$ is a $m_{i}$-dimensional standard sphere.

We can assume that there are precisely $g$ focal points $\gamma(\theta_{1}), \cdots, \gamma(\theta_{g})$ of $M$ along $\gamma$ after moving $p$ slightly in $\Sigma$ if necessary, where $\gamma(\theta_{i})\in l_{i}(q_{0})\cap \Sigma$ and $\theta_{g}< \theta_{g-1}< \cdots < \theta_{1}$. Clearly, $\theta_{1}< \theta_{0}\leq \pi$. By Theorem 5.2 in \cite{hpt} and theories of $\mathcal{F}$ and $\mathcal{\overline{F}}$, it is not hard to prove that $(Z, f)$ is an $\mathbb{Z}_{2}$-orientiable Bott-Samelson cycle at $q_{0}$, where $Z=\{(y_{1}, \cdots,y_{g})\mid y_{1}\in S_{1}(q_{0}), y_{2}\in S_{2}(y_{1}),\cdots, y_{g}\in S_{g}(y_{g-1}) \}$ is an iterated sphere bundle of dimension $m:=\dim M$, and $f: Z\rightarrow M$ is defined by $f(y_{1}, \cdots,y_{g})=y_{g}$. Since both $M$ and $Z$ are compact manifolds of dimension $m$, by Theorem 4.20 in \cite{hpt} $(Z,f)$ represents the nontrivial homology class in $H_m(M,\mathbb Z_2)$. It is well known that such a map $f$ must be surjective, since $M$ is compact. For $x\in M$ and $1\leq i\leq g-1$, we set $S_{i+1}\circ S_{i}\circ \cdots \circ S_{1}(x):=\{x_{i+1}|x_{i+1}\in S_{i+1}(x_{i}), x_{i}\in S_{i}\circ \cdots \circ S_{1}(x)\}$. Thus $$M=S_{g}\circ S_{g-1}\circ \cdots \circ S_{1}(q_{0}). \eqno{(1)}$$

We consider the parallel leaf $f_{\theta}: M\rightarrow L(\gamma(\theta))$ in $\mathbb{S}^{n}$ defined by $$f_{\theta}(x)=\cos \theta\ x + \sin \theta\ \xi (x).$$ Then for any $x\in M$ and $1\leq i\leq g$, $f_{\theta_{i}}(x)$ is a focal point of $M$ since we have $$f_{\theta_{i}}(S_{i}(x))=f_{\theta_{i}}(x). \eqno{(2)}$$

Consider now the family $\mathcal{B}$ of broken horizontal geodesics in $\mathbb{S}^{n}$ from $p=\gamma(\theta_{0})$ to $M=L(\gamma(0))$, whose projection (to the leaf space $\mathbb{S}^{n}/\mathcal{F}$) is the same as $\gamma$, that are allowed to change directions at the singular leaves. For $0\leq \theta\leq\theta_{0}$, define $\mathcal{B}(\theta):=\{c(\theta)\mid c\in \mathcal{B}\}$, which is a subset of $L(\gamma(\theta))$. Notice that we will consider $\mathcal{B}(\theta)$ in the direction from $\theta=\theta_{0}$ to $\theta=0$.
Then $$\mathcal{B}(\theta)=f_{\theta}(q_{0})=\gamma(\theta)\ \ for\ \ \theta\in (\theta_{1}, \theta_{0}].$$ By (2) we get that $f_{\theta_{1}}(q_{0})=f_{\theta_{1}}(S_{1}(q_{0}))$, thus
$$\mathcal{B}(\theta)=f_{\theta}(S_{1}(q_{0}))\ \ for\ \ \theta\in (\theta_{2}, \theta_{1}].$$
Hence, after $g$ steps, we get that $$\mathcal{B}(\theta)=f_{\theta}(S_{g}\circ S_{g-1}\circ \cdots \circ S_{1}(q_{0}))\ \ for\ \ 0\leq\theta\leq \theta_{g}.\eqno{(3)}$$

By (1), (3) we have $$\mathcal{B}(\theta)=f_{\theta}(M)=L(\gamma(\theta))\ \ for\ \ 0\leq\theta\leq \theta_{g}.$$
It follows that $\mathcal{B}(0)=L(q_{0})=L(q)$, which means that $p$ and $q$ can be joined by a broken horizontal geodesic of length $\theta_{0}$. Thus $diam_{\mathcal{H}}\mathbb{S}^{n}=\pi$ since $d_{\mathcal{H}}(p, q)\leq \theta_{0}\leq \pi$. $\hfill \square$

\section{Reducible actions on $\mathbb{S}^{n}$ and applications}
\renewcommand{\thesection}{\arabic{section}}
\renewcommand{\theequation}{\thesection.\arabic{equation}}
\setcounter{equation}{0}
\setcounter{theorem}{0}
For a singular Riemannian foliation $\mathcal{F}$ with closed leaves on a unit sphere $\mathbb{S}^{n}$, let $f:\mathbb{S}^{n}\rightarrow X=\mathbb{S}^{n}/\mathcal{F}$ be the project map. The quotient space $X$ has an induced metric: for any $x, y\in X$, we define $d(x, y)=d(f^{-1}(x), f^{-1}(y))$. Let $diam X$ denote the diameter of $X$ defined by this induced metric. By using theories in \cite{cg} we can get the following Lemma, which has been pointed out by Chen and Grove \cite{cg} (page 786) for the case of group actions:
\begin{lemma}\label{reducible}
Let $\mathcal{F}$ be a singular Riemannian foliation with closed leaves on a unit sphere $\mathbb{S}^{n}$ and $f:\mathbb{S}^{n}\rightarrow X=\mathbb{S}^{n}/\mathcal{F}$ be the project map, then $diam X\geq\frac{\pi}{2}$ (or equivalently $diam X=\frac{\pi}{2}$ or $\pi$) if and only if there is an integer $0\leq k\leq n-1$ and a unit sphere $\mathbb{S}^{k}\subset \mathbb{S}^{n}$ saturated by the leaves of $\mathcal{F}$, i.e., with $\mathbb{S}^{k}=f^{-1}(f(\mathbb{S}^{k}))$. In particularly, if $\mathcal{F}$ is given by the orbits of an isometric action of a compact Lie group $G$, then $diam X\geq\frac{\pi}{2}$ if and only if the action of $G$ is reducible.
\end{lemma}

\begin{proof} In Theorem 2.2 in \cite{sx} we proved that if $diam X\geq\frac{\pi}{2}$, then $diam X=\frac{\pi}{2}$ or $\pi$. Now we prove the `only if'
part. Suppose $x, y\in X$ such that $d(x,y)=diam X=\pi$, then by Proposition 4 in \cite{cg} $f^{-1}(x)=p$ and $f^{-1}(y)=q$ are point leaves in $\mathbb{S}^{n}$. Since $d(p, q)=d(x, y)=\pi$, $q$ is the antipodal point of $p$. Thus $\mathbb{S}^{0}=\{p, q\}$ is a unit sphere saturated by leaves. If $diam X=\frac{\pi}{2}$, then by Theorem 2.4 in \cite{sx} there is a unit sphere $\mathbb{S}^{k}$ saturated by leaves, where $0< k\leq n-1$.

Next we prove the `if' part. For $\mathbb{S}^{k}\subset \mathbb{S}^{n}$, there is a subsphere $\mathbb{S}^{n-k-1}\subset \mathbb{S}^{n}$ such that $\mathbb{S}^{n}$ is the spherical join of $\mathbb{S}^{k}$ and $\mathbb{S}^{n-k-1}$, so we have $d(\mathbb{S}^{k},\mathbb{S}^{n-k-1})\equiv \frac{\pi}{2}$, that is, $d(p,q)=\frac{\pi}{2}$ for any $p\in \mathbb{S}^{k}$ and $q\in \mathbb{S}^{n-k-1}$. For any $p\in \mathbb{S}^{k}$, $L(p)\subset \mathbb{S}^{k}$ since $\mathbb{S}^{k}$ is saturated by leaves. Now for any $q\in \mathbb{S}^{n-k-1}$, $d(f(p), f(q))=d(L(p), q)=\frac{\pi}{2}$. Thus $diam X\geq d(f(p), f(q))=\frac{\pi}{2}$.
\end{proof}

\noindent {\it Proof of Theorem \ref{cohom 2}.} If $G$ has cohomogeneity 1, then the action of $G$ is polar. Next assume that $G$ has cohomogeneity 2 and let $\dim G$ denote the dimension of $G$. If $\dim G=1$, then the action of $G$ is reducible since $G$ is a commutative group. Now assume $\dim G>1$ and the action of $G$ is not polar. By Theorem 5.1 in \cite{st} there is a 1-dimensional linear group $K$ such that $\mathbb{S}^{n}/G=\mathbb{S}^{3}/K$. Now by Lemma \ref{reducible} $diam (\mathbb{S}^{n}/G)=diam (\mathbb{S}^{3}/K)\geq\frac{\pi}{2}$ since the action of $K$ is reducible. Using Lemma \ref{reducible} again we get that the action of $G$ is reducible.
$\hfill \square$\\

\noindent {\it Proof of Theorem \ref{infinitesimally polar}.} Our proof is an application of Theorem \ref{GL}. By Theorem \ref{cohom 2} and Lemma \ref{reducible} any action belonged to case (1), (3) and (4) in Theorem \ref{GL} is either polar or reducible. Meanwhile, any action belonged to case (2) in Theorem \ref{GL} gives rise to a Riemannian foliation. For a singular Riemannian foliation $\mathcal{F}$ with closed leaves on $\mathbb{S}^{n}$, we showed in \cite{sx} that if $\mathcal{F}$ is a Riemannian foliation or if $\mathcal{F}$ is given by the orbits of a reducible action, then $diam_{\mathcal{H}}\mathbb{S}^{n}=\pi$.  Thus combining Theorem \ref{1.1} with above arguments, we get the proof of Theorem \ref{infinitesimally polar}.
$\hfill \square$

 \end{document}